\documentclass{article}
\usepackage{graphicx}
\usepackage[margin=1in]{geometry}
\usepackage[T1]{fontenc}
\usepackage{times}
\usepackage{authblk}
\usepackage{amsmath,amsfonts,amssymb,amsthm}
\usepackage{hyperref,cleveref,doi}
\hypersetup{colorlinks=true, breaklinks=true, urlcolor=black, linkcolor=black, citecolor=black}
\usepackage{enumerate,enumitem,csquotes,stmaryrd}
\usepackage{color,todonotes,soul,comment}

\newtheorem{theorem}{Theorem}[section]
\newtheorem{lemma}[theorem]{Lemma}
\newtheorem{proposition}[theorem]{Proposition}
\newtheorem{corollary}[theorem]{Corollary}

\theoremstyle{remark}

\newcommand{\bN}{\mathbb{N}}
\newcommand{\bR}{\mathbb{R}}
\newcommand{\cO}{\mathcal{O}}
\newcommand{\cv}[1]{\underset{#1}{\longrightarrow}}

\newcommand{\Landau}[3]{\underset{#2}{#1}\left(#3\right)} 
\newcommand{\abs}[1]{\left|#1\right|}

\newcommand{\dist}{\mathrm{dist}}

\newcommand{\bE}{\mathbb{E}}
\newcommand{\expec}[1]{\bE\left[#1\right]}

\newcommand{\bP}{\mathbb{P}}
\newcommand{\prob}[1]{\bP\left(#1\right)}

\newcommand{\Leb}{\mathrm{Leb}}

\newcommand{\LIS}[1]{\mathrm{LIS}\left(#1\right)}

\newcommand{\carre}{[0,1]^2}

\newcommand{\perm}[1]{\mathrm{Perm}\left(#1\right)}
\newcommand{\sample}[2]{\mathrm{Sample}_{#1}\left(#2\right)}

\newcommand{\LISt}{\widetilde{\mathrm{LIS}}}
\def\OT{\widetilde{\cO}}
\def\ThetaT{\widetilde{\Theta}}
\def\OmegaT{\widetilde{\Omega}}

\title{Locally uniform random permutations with large increasing subsequences}

\author[1]{Victor Dubach}
\affil[1]{Institut {\'E}lie Cartan de Lorraine, Universit{\'e} de Lorraine, Nancy, France}

\date{}

\begin{document}

\maketitle

\begin{abstract}
	We investigate the maximal size of an increasing subset among points randomly sampled from certain probability densities.
    Kerov and Vershik's celebrated result states that the largest increasing subset among $N$ uniformly random points on $[0,1]^2$ has size asymptotically $2\sqrt{N}$.
    More generally, the order $\Theta(\sqrt{N})$ still holds if the sampling density is continuous.
    In this paper we exhibit two sufficient conditions on the density to obtain a growth rate equivalent to any given power of $N$ greater than $\sqrt{N}$, up to logarithmic factors.
    Our proofs use methods of slicing the unit square into appropriate grids, and investigating sampled points appearing in each box.
\end{abstract}

\section{Introduction}

\subsection{Random permutations sampled from a pre-permuton}

We start by defining the model of random permutations studied in this paper.
Consider points $X_1,\dots,X_N$ in the unit square $\carre$ whose $x$-coordinates and $y$-coordinates are all distinct. 
One can then define a permutation $\sigma$ of size $N$ in the following way: 
for any $i,j\in\llbracket1,N\rrbracket$, let $\sigma(i)=j$ whenever the point with $i$-th lowest $x$-coordinate has $j$-th lowest $y$-coordinate. 
We denote by $\perm{X_1,\dots,X_N}$ this permutation; 
see \Cref{fig_perm_ex} for an example.
Now suppose $\mu$ is a probability measure on $\carre$ and $X_1,\dots,X_N$ are random i.i.d.~points distributed under $\mu$: 
the random permutation $\perm{X_1,\dots,X_N}$ is then denoted by $\sample{N}{\mu}$.
To ensure this permutation is well defined, we suppose that the marginals of $\mu$ have no atom so that $X_1,\dots,X_N$ have almost surely distinct $x$-coordinates and $y$-coordinates.
We call such a measure a \textit{pre-permuton}; 
see \Cref{section_discussions} for a discussion around this name.
\medskip

Note that permutations sampled from the uniform measure on $\carre$ are uniformly random. 
The model of random permutations previously defined thus generalizes the uniform case while allowing for new tools in a geometric framework, as illustrated in \cite{AD95} (see also \cite{K06} for a variant with uniform involutions). 
This observation motivates the study of such models, as done for example in \cite{DZ95} or \cite{S22}.
\medskip

In the present paper we are interested in pre-permutons that are absolutely continuous with respect to the Lebesgue measure on $\carre$, and denote by $\mu_\rho$ the pre-permuton having density $\rho$.
Following \cite{S22} we call the permutations sampled under $\mu_\rho$ locally uniform.
This name is easily understood when $\rho$ is continuous, since the measure $\mu_\rho$ can then locally be approximated by a uniform measure.

\subsection{Growth speed of the longest increasing subsequence}

\begin{figure}
    \centering
    \includegraphics[scale=0.36]{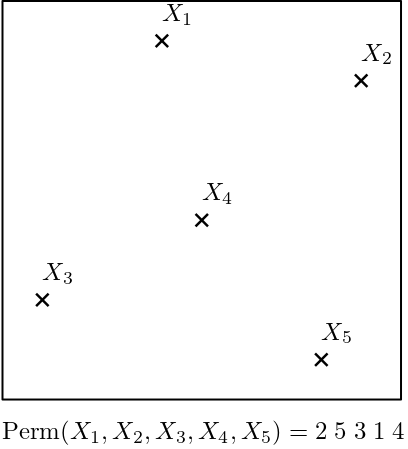}
    \caption{A family of points and its associated permutation, written in one-line notation $\sigma = \sigma(1)\, \sigma(2) \dots \sigma(N)$. 
    Here we have $\sigma(1)=2$ because the leftmost point is second from the bottom; and so on.}
    \label{fig_perm_ex}
\end{figure}

Let $\sigma$ be a permutation of size $N$. 
An increasing subsequence of $\sigma$ is a sequence of indices $i_1<\dots<i_k$ such that $\sigma(i_1)<\dots<\sigma(i_k)$. 
The maximal length of such a sequence is called the size of the \textit{longest increasing subsequence} of $\sigma$ and denoted by $\LIS{\sigma}$.
Let us write (here and throughout this paper), for all $N\in\bN^*$,
\begin{equation}\label{def_ell_N}
    \ell_N := \expec{\LIS{\sigma_N}}
    \:\text{ where }\:
    \sigma_N \text{ is a uniformly random permutation of size } N .
\end{equation}
In the 60's, Ulam asked about the asymptotic behavior of $\ell_N$ as $N\to\infty$.
The study of longest increasing subsequences has since then been a surprisingly fertile research subject with unexpected links to diverse areas of mathematics; 
see \cite{R15} for a review. 
A solution to Ulam's problem was found by Vershik and Kerov; 
using Young diagrams and the Robinson-Schensted's correspondence, they obtained the following:

\begin{theorem}[\cite{KV77}]\label{Ulam}
For each integer $N$, let $\sigma_N$ be a uniform permutation of size $N$. Then:
\begin{equation*}
    \frac{1}{\sqrt{N}} \LIS{\sigma_N} \cv{N\to\infty} 2
\end{equation*}
in probability and $L^1$-norm.
In particular, writing $\ell_N := \expec{\LIS{\sigma_N}}$:
\begin{equation*}
    \frac{1}{\sqrt{N}}\ell_N \cv{N\to\infty} 2.
\end{equation*}
\end{theorem}

The asymptotic behavior of the longest increasing subsequence in the uniform case is now well understood with concentration inequalities \cite{F98,T01} and an elegant asymptotic development \cite{BDJ99}.
Note that the concentration inequalities later recalled in \Cref{th_McDiarmid,th_Talagrand} can be used to recover \Cref{Ulam} from the first moment convergence.
\smallskip

It is then natural to try and generalize \Cref{Ulam} to $\LIS{\sample{N}{\mu}}$ for appropriate pre-permutons $\mu$.
One of the first advances on this question was obtained by Deuschel and Zeitouni who proved:
\begin{theorem}[\cite{DZ95}, Theorem 2]\label{Deuschel_Zeitouni}
    If $\rho$ is a $C_b^1$, bounded below probability density on $\carre$ then:
    \begin{equation*}
        \frac{1}{\sqrt{N}} \LIS{\sample{N}{\mu_\rho}} \cv{N\to\infty} K_\rho    
    \end{equation*}
    in probability, for some positive constant $K_\rho$ defined by a variational problem.
\end{theorem}

This $\sqrt{N}$ behavior holds more generally when the sampling density is continuous, as we prove in \Cref{section_majoration_dens}:

\begin{proposition}\label{coro_cont}
    Let $f$ be a continuous probability density on $\carre$. Then:
    \begin{equation*}
        \expec{\LIS{\sample{N}{\mu_f}}} = \Landau{\Theta}{N\to\infty}{\sqrt{N}}.
    \end{equation*}
\end{proposition}

These results, as well as most of the literature on the subject, are restricted to the case of a pre-permuton with \enquote{regular}, bounded density.
The goal of this paper is to investigate the asymptotic behavior of $\LIS{\sample{N}{\mu_\rho}}$ when $\rho$ is a probability density on $\carre$ satisfying certain types of divergence. 
We state in \Cref{section_theoremes} sufficient conditions on $\rho$ for the quantity $\expec{\LIS{\sample{N}{\mu_\rho}}}$ to be equivalent to any given power of $N$ (between $N^{1/2}$ and $N$), up to logarithmic factors.
We then present in \Cref{section_concentration} a few concentration inequalities for $\LIS{\sample{N}{\mu_\rho}}$, explaining why we can focus on the asymptotic behavior of the mean.
\medskip

Similar asymptotics of LIS have also been found for other models of random permutations.
This is the case in \cite[Theorem 1.2]{BP15}, for Mallows random permutations under certain regimes of parameters.
However such models are quite different from the sampled permutations studied here: 
in the regime $n(1-q_n) \to +\infty$ of \cite{BP15}, Mallows permutations converge to the permuton that puts uniform mass along the diagonal of the unit square.
Let us also mention the power-law bounds obtained in \cite[Theorem 1.1]{BDG24} for permutations sampled under the (biased) \textit{Brownian separable permuton}.
This random permuton is defined as the push-forward of the Lebesgue measure via a mapping related to the Brownian excursion, and the techniques of \cite{BDG24} hinge the analysis of a fragmentation process.
\smallskip

Lastly, it might be worth pointing out that growth rates found in this paper can be seen as \enquote{intermediate} in the theory of pre-permutons. 
Indeed, we previously explained how the $\sqrt{N}$ behavior corresponds to a \enquote{regular} case. 
In another paper we study conditions under which the sampled permutation's LIS grows linearly with $N$:
\begin{proposition}[\cite{D24}]
    Let $\mu$ be a pre-permuton and define
    \begin{equation*}
        \LISt(\mu) := \max_A \mu(A)
    \end{equation*}
    where the maximum is taken over all \enquote{increasing} subsets of $\carre$, in the sense that any pair of points is $\prec$-ordered with the notation of \Cref{section_notations}.
    Then the function $\LISt$ is upper semi-continuous on pre-permutons and satisfies
    \begin{equation*}
        \frac{1}{N}\LIS{\sample{N}{\mu}} \cv{N\to\infty} \LISt(\mu)
        \quad\text{almost surely}.
    \end{equation*}
\end{proposition}

\section{Our results}

\subsection{Some notation}\label{section_notations}

Throughout the paper, the only order on the plane we consider is the partial order $\prec$ defined by:
\begin{equation*}
    \text{for all } (x_1,y_1) , (x_2,y_2) \in \bR^2 ,\quad
    (x_1,y_1)\prec(x_2,y_2) 
    \;\text{ if and only if }\;
    x_1<x_2 \text{ and } y_1<y_2.
\end{equation*}
We also write $\dist$ for the $L^1$-distance in the plane, namely:
\begin{equation*}
    \text{for all } (x_1,y_1) , (x_2,y_2) \in \bR^2 ,\quad
    \dist\big((x_1,y_1),(x_2,y_2)\big) := \abs{x_1 - x_2} + \abs{y_1 - y_2} ,
\end{equation*}
and we denote by $\Delta$ the diagonal of the unit square $\carre$.
We use the symbols $\bN$ for the set of non-negative integers, and $\bN^*$ for the set of positive integers.
\bigskip

Consider points $X_1,\dots,X_N$ in the unit square $\carre$, with distinct $x$-coordinates and distinct $y$-coordinates. 
Then the quantity $\LIS{\perm{X_1,\dots,X_N}}$ is easily read on the visual representation: 
it is the maximum size of an \enquote{increasing} subset of these points, \textit{i.e.}~the maximum number of points forming an up-right path. 
For this reason and to simplify notation, we write $\LIS{X_1,\dots,X_N}$ for this quantity.
\bigskip

Let $(a_n),(b_n)$ be two sequences of positive real numbers.
We write $a_n\sim b_n$ when they are asymptotically equivalent, that is when $a_n/b_n \to 1$.
We use the symbols $\OT, \ThetaT, \OmegaT$ for asymptotic comparisons up to logarithmic factors:
write $a_n = \OT(b_n)$ as $n\to\infty$ when there exist constants $c_1>0$ and $c_2\in\bR$ such that for some integer $n_0$:
\begin{equation*}
    \text{for all } n\ge n_0,\quad a_n \le c_1 \log(n)^{c_2}b_n.
\end{equation*}
We also write $a_n=\OmegaT(b_n)$ when $b_n=\OT(a_n)$, and $a_n=\ThetaT(b_n)$ when simultaneously $a_n = \OT(b_n)$ and $a_n=\OmegaT(b_n)$.
When these comparisons hold with no logarithmic factor (\textit{i.e.}~$c_2=0$), we use the standard notation $\cO,\Theta,\Omega$.

\subsection{First moment asymptotics of the longest increasing subsequence}\label{section_theoremes}

Our main results are two conditions on the divergence of the pre-permuton density that imply a large growth rate for the longest increasing subsequences in the sampled permutations. 
First we study densities diverging at a single point (see the left-hand side of \Cref{plot_densites}) and then we study densities diverging along the diagonal (see the right-hand side of \Cref{plot_densites}).

\begin{figure}
    \centering
    \includegraphics[scale=0.49]{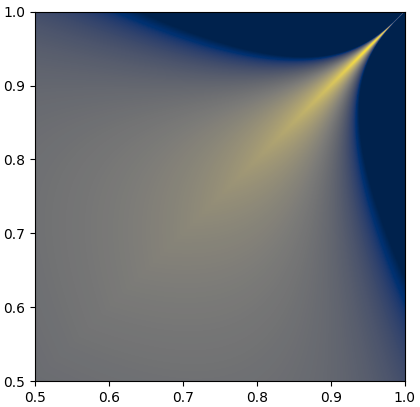}
    \qquad
    \quad
    \includegraphics[scale=0.67]{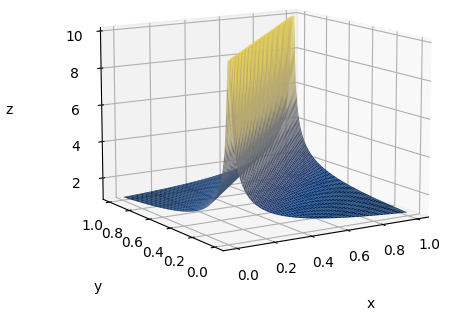}
    \caption{Representation of the divergent densities studied in this paper. 
    On the left, a representation of the density appearing in \Cref{ODG_divergence_pinched} with $\frac{\beta}{\beta-1}=3$. 
    Bright yellow indicates a high value while dark blue indicates a low value.
    On the right, the 3D graph of a function satisfying the hypothesis of \Cref{ODG_divergence_along} with $\alpha=-0.5$.}
    \label{plot_densites}
\end{figure}

The first natural type of divergence to consider is a divergence at a single point.
Suppose this happens at the north-east corner, in a radial way around this point. We show in this case that longest increasing subsequences behave similarly to the continuous density case, up to a logarithmic factor.

\begin{theorem}\label{ODG_divergence_integrable}
    Suppose the density $\rho$ is continuous on $\carre\setminus\lbrace (1,1)\rbrace$ and satisfies
    \begin{equation*}
        \rho(x,y) = \Theta\left(d^\alpha\right)
        \quad\text{as}\quad (x,y) \to (1,1)
    \end{equation*}
    where $d:=\dist((x,y),(1,1))$, for some $\alpha > -2$. Then:
    \begin{equation*}
        \expec{\LIS{\sample{N}{\mu_\rho}}}
        = \Landau{\ThetaT}{N\to\infty}{\sqrt N} .
    \end{equation*}
\end{theorem}

Note that the condition $\alpha>-2$ is necessary for integrability.
In order to see long increasing subsequences appear in the sampled permutations, we can "pinch" the density along the diagonal when approaching the north-east corner. 
This will force sampled points to concentrate along the diagonal, thus likely forming increasing subsequences, and allow for sharper divergence exponents.

\begin{theorem}\label{ODG_divergence_pinched}
Suppose the density $\rho$ is continuous on $\carre\setminus\lbrace (1,1)\rbrace$ and satisfies
\begin{align*}
    \rho(x,y) = \Theta\left(
    d^{\frac{\beta}{1-\beta}}\exp\left(
    -c \abs{x-y} d^{\frac{\beta}{1-\beta}}
    \right) \right)
    \quad\text{as}\quad (x,y) \to (1,1)
\end{align*}
where $d:= \dist((x,y),(1,1))$, for some $\beta \in ]1,2[$ and $c>0$. 
Then:
\begin{equation*}
    \expec{\LIS{\sample{N}{\mu_\rho}}}
    = \Landau{\ThetaT}{N\to\infty}{N^{1/\beta}} .
\end{equation*}
\end{theorem}

Note that when $\beta$ varies between $1$ and $2$, the exponent $\frac{\beta}{1-\beta}$ varies between $-\infty$ and $-2$.
Note also that such densities exist by integrability of the estimate.
\medskip

Instead of a divergence at a single point, we may also study a type of divergence along an increasing curve.
This can be done with a power function and provides a different condition from \Cref{ODG_divergence_pinched} to obtain a behavior equivalent to any given power of $N$ (between $N^{1/2}$ and $N$), up to a logarithmic factor.

\begin{theorem}\label{ODG_divergence_along}
Suppose the density $\rho$ is continuous on $\carre\setminus\Delta$ and satisfies
\begin{equation*}
    \rho(x,y) = \Theta\left( 
    |x-y|^{\alpha}
    \right)
    \quad\text{as}\quad |x-y| \to 0
\end{equation*}
for some $\alpha \in ]-1, 0[$. Then:
\begin{equation*}
    \expec{\LIS{\sample{N}{\mu_\rho}}}
    = \underset{N\to\infty}{\ThetaT}\left( N^{1/(\alpha+2)} \right).
\end{equation*}
\end{theorem}

While these previous assumptions are all quite intuitive to consider,
\Cref{ODG_divergence_along} is of a somewhat other nature than \Cref{ODG_divergence_integrable,ODG_divergence_pinched}.
As explained in \Cref{section_methodes}, their proofs also work differently and illustrate slightly distinct techniques which might have broader applications.\bigskip

The study of densities in \Cref{ODG_divergence_integrable,ODG_divergence_pinched} relies on
a family of reference pre-permutons (\textit{permutons} actually, see \Cref{section_discussions}) that we now introduce.
Fix two parameters $\beta >1$ and $\gamma\in\bR$. 
Define for any positive integer $k\ge 1$:
\begin{equation*}
    u_{k} := \frac{1}{Z_{\beta,\gamma}}k^{-\beta}\log(k+1)^\gamma
    \quad\text{where}\quad
    Z_{\beta,\gamma} := \sum_{k\ge 1}k^{-\beta}\log(k+1)^\gamma.
\end{equation*}
For all $n\ge 0$, set
$S_{n} := \sum_{k=1}^n u_{k}$
and consider the sequence of disjoint boxes
$C_{n} := [S_{n-1},S_{n}]^2$, $n\in\bN^*$,
covering the diagonal in an up-right manner.
We can then define a probability density on the unit square by
\begin{equation*}
    \rho_{\beta,\gamma}^\nearrow := \sum_{k\ge 1}
    u_{k}^{-1}\mathbf{1}_{C_{k}}
\end{equation*}
and we write $\mu_{\beta,\gamma}^\nearrow$ for the (pre-)permuton having density $\rho_{\beta,\gamma}^\nearrow$ with respect to Lebesgue measure on $\carre$. See \Cref{fig_reference} for a representation.

\begin{figure}
\begin{center}
\includegraphics[scale=0.18]{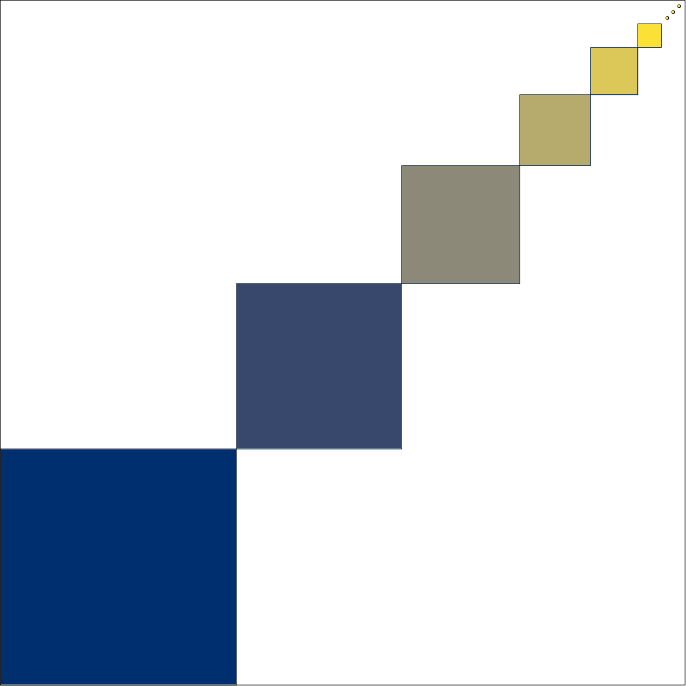}
\caption{A representation of the permuton $\mu_{\beta,\gamma}^\nearrow$. 
Bright yellow indicates a high value of the density while dark blue indicates a low value of the density.}\label{fig_reference}
\end{center}
\end{figure}

\begin{proposition}\label{ODG_reference}
Let $\gamma\ge0$. If $\beta\in ]1,2[$ then:
\begin{equation*}
    \expec{\LIS{\sample{N}{\mu_{\beta,\gamma}^\nearrow}}}
    = \underset{N\to\infty}{\ThetaT}\left( N^{1/\beta} \right).
\end{equation*}
If $\beta\ge 2$ then:
\begin{equation*}
    \expec{\LIS{\sample{N}{\mu_{\beta,\gamma}^\nearrow}}}
    = \underset{N\to\infty}{\ThetaT}\left( \sqrt{N} \right).
\end{equation*}
\end{proposition}

In \Cref{section_last_proofs} we use \Cref{ODG_reference} to prove \Cref{ODG_divergence_integrable,ODG_divergence_pinched} by comparing the densities involved.
The parameter $\beta$ allows $\rho_{\beta,\gamma}^\nearrow$ to have similar asymptotics to the density $\rho$ in \Cref{ODG_divergence_integrable} or \ref{ODG_divergence_pinched}, while the parameter $\gamma$ allows to slightly modify the size of the boxes so that they cover the divergence of $\rho$ adequately.
When $\gamma=0$, we drop this subscript and simply write $\rho_{\beta}^\nearrow$.

\subsection{Concentration around the mean}\label{section_concentration}

In this paper we only investigate the mean of $\LIS{\sample{N}{\mu}}$. 
The reason for this is that we can easily deduce asymptotic knowledge of the random variable itself from well known concentration inequalities. 
In our case it is sufficient to use what is usually referred to as Azuma's or McDiarmid's inequality, found in \cite[Lemma 1.2]{McD89} and whose origin goes back to \cite{A67}. 
One of its most common use is for the chromatic number of random graphs, but it is also well adapted to the study of longest increasing subsequences as illustrated in \cite{F98}. 

\begin{theorem}[McDiarmid's inequality]\label{th_McDiarmid}
Let $N\in\bN^*$, $X_1,\dots,X_N$ be independent random variables with values in a common space $\mathcal{X}$.
Let $f:\mathcal{X}^N\to\bR$ be a function satisfying the bounded differences property, that is: 
for all $i\in\llbracket1,N\rrbracket$ and $x_1,\dots,x_N,y_i\in\mathcal{X}$,
\begin{equation*}
    \abs{ f(x_1,\dots,x_i,\dots,x_N) - f(x_1,\dots,y_i,\dots,x_N) } \le c
\end{equation*}
for some constant $c>0$. 
Then for any positive number $\lambda>0$:
\begin{equation*}
    \prob{\big. \abs{ f(X_1,\dots,X_N) - \expec{f(X_1,\dots,X_N)} } > \lambda }
    \le 2\exp\left( \frac{-2\lambda^2}{N c^2} \right).
\end{equation*}
\end{theorem}

We can apply this to $\LIS{X_1,\dots,X_N}$ where $X_1,\dots,X_N$ are i.i.d.~points distributed under $\mu$, noticing that changing the value of a single point changes the size of the largest increasing subset by at most $1$.

\begin{corollary}\label{Azuma}
Let $\mu$ be a pre-permuton. 
Then for any $N\in\bN^*$ and $\lambda>0$:
\begin{equation*}
    \prob{\big. \abs{ \LIS{\sample{N}{\mu}} - \expec{\LIS{\sample{N}{\mu}}} } > \lambda }
    \le 2\exp\left( \frac{-2\lambda^2}{N} \right).
\end{equation*}
\end{corollary}

This concentration inequality is especially useful when $\expec{\LIS{\sample{N}{\mu}}}$ is of order greater than $\sqrt{N}$, which is for example the case in \Cref{ODG_divergence_pinched} when $\beta\in]1,2[$.
\Cref{Azuma} then implies that the variable is concentrated around its mean in the sense that
\begin{equation*}
    \frac{\LIS{\sample{N}{\mu}}}{\expec{\LIS{\sample{N}{\mu}}}} \cv{N\to\infty} 1
\end{equation*}
in probability. 
Moreover $\LIS{\sample{N}{\mu}}$ admits a median of order $\ThetaT\left( N^{1/\beta} \right)$, and an analogous remark holds for \Cref{ODG_divergence_along}.
One could then apply the following sharper concentration inequality:

\begin{theorem}[Talagrand's inequality for longest increasing subsequences]\label{th_Talagrand}
Let $\mu$ be a pre-permuton. 
For any $N\in\bN^*$, denote by $M_N$ a median of $L_N:=\LIS{\sample{N}{\mu}}$. 
Then for all $\lambda>0$:
\begin{align*}
    \prob{ L_N \ge M_N + \lambda }
    \le 2\exp\left( \frac{-\lambda^2}{4(M_N+\lambda)} \right)
    \quad\text{and}\quad
    \prob{ L_N \le M_N - \lambda }
    \le 2\exp\left( \frac{-\lambda^2}{4M_N} \right).
\end{align*}
\end{theorem}
See \cite[Theorem 7.1.2]{T01} for the original reference in the case of uniform permutations. 
The proof works the same for random permutations sampled from pre-permutons.
See also \cite[Theorem 5]{K06} for a nice application to longest increasing subsequences in random involutions.

\subsection{Discussion}\label{section_discussions}

\textit{Improvements.}

Several hypotheses made in the theorems simplify the calculations but are not crucial to the results. 
For instance \Cref{ODG_divergence_integrable,ODG_divergence_pinched} could be generalized by replacing the north-east corner with any point in the unit square and the diagonal with any local increasing curve passing through that point, under appropriate hypotheses.
A similar remark holds for \Cref{ODG_divergence_along}.
We could also state \Cref{ODG_reference} for general $\gamma\in\bR$, but prefer restricting ourselves to the case $\gamma\ge0$ since this is all we need for the proofs of \Cref{ODG_divergence_integrable,ODG_divergence_pinched} and it requires a bit less work.

The necessity of logarithmic factors in our estimates remains an open question. 
We believe our results could be sharpened in this direction, but our techniques do not seem sufficient to this aim.
\bigskip

\noindent\textit{Links to permuton and graphon theory.}

When $\mu$ is a probability measure on $\carre$ whose marginals are uniform, we call it a \textit{permuton} as in \cite{GGKK15}. 
The theory of permutons was introduced in \cite{HKMRS13} and is now widely studied \cite{KP12,M15,BBFGP18,BBFGMP22}. 
It serves as a scaling limit for random permutations and is directly related to the study of pattern occurrences (see e.g.~\cite[Definition 1.5]{HKMRS13} or \cite[Theorem 2.5]{BBFGMP20}).
One of its fundamental results is that for any permuton $\mu$, the sequence $\big(\sample{N}{\mu}\big)_{N\in\bN^*}$ almost surely converges \enquote{in the permuton sense} to $\mu$.

Reading this paper does not require any prior knowledge about the literature on permutons: 
it is merely part of our motivation for the study of models $\sample{N}{\mu}$.
Notice however that considering pre-permutons instead of permutons is just a slight generalization. 
Indeed, one can associate to any pre-permuton $\mu$ a unique permuton $\hat{\mu}$ such that random permutations sampled from $\mu$ and $\hat{\mu}$ have the same law (see e.g.~\cite[Remark 1.2]{BDMW22}).

This paper was partly motivated by \cite{McK19}, where an analogous problem is tackled for graphons. 
The theory of graphons for the study of dense graph sequences is arguably the main inspiration at the origin of permuton theory, and there exist numerous bridges between them \cite{GGKK15,BBDFGMP21}. 
For instance the longest increasing subsequence of permutations corresponds to the clique number of graphs. 
In \cite{DHM15} the authors exhibit a wide family of graphons bounded away from $0$ and $1$ whose sampled graphs have logarithmic clique numbers, thus generalizing this property of Erd\H{o}s-Rényi random graphs. 
In some sense this is analogous to Deuschel and Zeitouni's result on permutations (\Cref{Deuschel_Zeitouni} here). 
In \cite{McK19} the author studies graphons allowed to approach the value $1$, and proves in several cases that clique numbers behave as a power of $N$; 
the results of the present paper are counterparts for permutations.

\subsection{Proof method and organization of the paper}\label{section_methodes}

The proofs of \Cref{ODG_divergence_integrable,ODG_divergence_pinched} rely on bounding the density of interest on certain appropriate areas with other densities which are easier to study. 
This general technique is developed in \Cref{section_majoration_dens} where we prove two lemmas of possible independent interest. 

\Cref{section_ref} is devoted to our reference permutons, which are the main ingredient when bounding general densities. 
The idea for the proof of \Cref{ODG_reference} is that points sampled from $\mu_{\beta,\gamma}^\nearrow$ are uniformly sampled on each box $C_n$. 
We can thus use \Cref{Ulam} on each box containing enough points, the latter property being studied through appropriate concentration inequalities on binomial variables.

We then prove \Cref{ODG_divergence_integrable,ODG_divergence_pinched} in \Cref{section_last_proofs}, using all the previously developed tools.

Finally, we prove \Cref{ODG_divergence_along} in \Cref{section_new_proofs}. 
This proof does not use the previous techniques and rather uses a grid on the unit square that gets thinner as $N\to\infty$. 
The main idea is to bound the number of points appearing in any increasing sequence of boxes.
The sizes of the boxes are chosen so that a bounded number of points appear in each box, and concentration inequalities are used to make sure such approximations hold simultaneously on every box.

\section{Bounds on LIS from bounds on the density}\label{section_majoration_dens}

One of the main ideas for the proofs of \Cref{ODG_divergence_integrable,ODG_divergence_pinched} is to deduce bounds on the order of LIS from bounds on the sampling density.
We state here two useful lemmas to this aim.

\begin{lemma}\label{lem_maj_density}
Suppose $f,g$ are two probability densities on $\carre$ such that $f\ge \varepsilon g$ for some $\varepsilon > 0$. Then:
\begin{equation*}
    \expec{ \LIS{\sample{N}{\mu_f}} } = \underset{N\to\infty}{\Omega}\left( \expec{ \LIS{\sample{\lfloor\varepsilon N\rfloor}{\mu_g}}} \right) .
\end{equation*}
Likewise, if $f\le Mg$ for some $M>0$ then
\begin{equation*}
    \expec{ \LIS{\sample{N}{\mu_f}} } = \underset{N\to\infty}{\cO}\left( \expec{ \LIS{\sample{\lceil M N\rceil}{\mu_g}} } \right) .
\end{equation*}
\end{lemma}

\begin{proof}
Let us deal with the first assertion of the lemma. We can write
\begin{equation*}
    f = \varepsilon g + (1-\varepsilon)h
\end{equation*}
for some other probability density $h$ on the unit square. 
The idea is to use a coupling between those densities. 
Let $N\in\bN^*$ and $B_1,\dots,B_N$ be i.i.d.~Bernoulli variables of parameter $\varepsilon$, $Y_1,\dots,Y_N$ be i.i.d.~random points distributed under density $g$, and $Z_1,\dots,Z_N$ be i.i.d.~random points distributed under density $h$, all independent. 
Then define for all $i$ between $1$ and $N$:
\begin{equation*}
    X_i := Y_i B_i + Z_i (1-B_i).
\end{equation*}
It is clear that $X_1,\dots,X_N$ are distributed as $N$ i.i.d.~points under density $f$. 
Let $I$ be the set of indices $i$ for which $B_i=1$. Then
\begin{equation*}
    \LIS{X_1,\dots,X_N} \ge \LIS{Y_i,i\in I} .
\end{equation*}
Hence, if $S_N$ denotes an independent binomial variable with parameter $(N,\varepsilon)$:
\begin{equation*}
    \expec{\LIS{\sample{N}{\mu_f}}} \ge
    \expec{\LIS{\sample{S_N}{\mu_g}}} \ge
    \expec{\LIS{\sample{\lfloor \varepsilon N \rfloor}{\mu_g}}}
    \prob{S_N \ge \varepsilon N}
\end{equation*}
where the latter is bounded away from $0$. 
This concludes the proof of the first assertion. 
The second one is a simple rewriting of it.
\end{proof}

\begin{lemma}\label{lem_maj_sum_density}
Suppose $f,g,h$ are probability densities on $\carre$ such that $f \le c_1 g+ c_2 h$ for some $c_1, c_2>0$. Then
\begin{equation*}
    \expec{\LIS{\sample{N}{\mu_f}}} = \underset{N\to\infty}{\cO}\left( 
    \expec{\LIS{\sample{\lceil M N\rceil}{\mu_g}}} + \expec{\LIS{\sample{\lceil M N\rceil}{\mu_h}}}
    \right)
\end{equation*}
for some constant $M>0$.
\end{lemma}

\begin{proof}
First write $c_1 g + c_2 h = M (\lambda g + (1-\lambda)h)$ with appropriate $M>0$ and $\lambda\in ]0,1[$. Applying \Cref{lem_maj_density} gives us:
\begin{equation*}
    \expec{ \LIS{\sample{N}{\mu_f}} } = \underset{N\to\infty}{\cO}\left( \expec{ \LIS{\sample{\lceil M N\rceil}{\mu_{\lambda g + (1-\lambda)h}}} } \right) .
\end{equation*}
We once again use a coupling argument. 
Let $N\in\bN^*$ and $B_1,\dots,B_{\lceil M N\rceil}$ be i.i.d.~Bernoulli variables of parameter $\lambda$, $Y_1,\dots,Y_{\lceil M N\rceil}$ be i.i.d.~random points distributed under density $g$, and $Z_1,\dots,Z_{\lceil M N\rceil}$ be i.i.d.~random points distributed under density $h$, all independent. 
Then define for all integer $i$ between $1$ and $\lceil M N\rceil$:
\begin{equation*}
    X_i := Y_i B_i + Z_i (1-B_i).
\end{equation*}
It is clear that $X_1,\dots,X_{\lceil M N\rceil}$ are distributed as $\lceil M N\rceil$ i.i.d.~points under density $\lambda g + (1-\lambda)h$. 
Moreover
\begin{equation*}
    \LIS{X_1,\dots,X_{\lceil M N\rceil}} 
    \le \LIS{Y_1,\dots,Y_{\lceil M N\rceil}} + \LIS{Z_1,\dots,Z_{\lceil M N\rceil}}
\end{equation*}
whence
\begin{equation*} 
    \expec{ \LIS{\sample{\lceil M N\rceil}{\mu_{\lambda g + (1-\lambda)h}}} } 
    \le \expec{ \LIS{\sample{\lceil M N\rceil}{\mu_{g}}} } 
    + \expec{ \LIS{\sample{\lceil M N\rceil}{\mu_{h}}} } .
\end{equation*}
This concludes the proof.
\end{proof}

Before moving on, we explain how to deduce \Cref{coro_cont} from \Cref{lem_maj_density}.

\begin{proof}[Proof of \Cref{coro_cont}]
    Since $f$ is continuous on $\carre$, there exists $M>0$ satisfying $f\le M$. 
    Using \Cref{Ulam} and \Cref{lem_maj_density} we get:
    \begin{equation*}
        \expec{ \LIS{\sample{N}{\mu_f}} } 
        = \cO\left( \ell_{\lceil MN \rceil} \right) 
        = \cO\left( \sqrt{N} \right)
    \end{equation*}
    as $N\to\infty$, where $\ell$ is defined in (\ref{def_ell_N}). 
    Then, $f$ also being non-zero, there exists $\varepsilon>0$ and a square box $C$ contained in $\carre$ such that $f\ge\varepsilon$ on $C$. 
    Since random points uniformly sampled in $C$ yield uniformly random permutations, \Cref{Ulam} and \Cref{lem_maj_density} imply:
    \begin{equation*}
        \expec{ \LIS{\sample{N}{\mu_f}} }
        = \Omega\left( \expec{ \LIS{\sample{\lfloor\varepsilon \Leb(C)N\rfloor}{\Leb_C}} } \right)
        = \Omega\left( \ell_{\lfloor\varepsilon \Leb(C)N\rfloor} \right) 
        = \Omega\left( \sqrt{N} \right)
    \end{equation*}
    as $N\to\infty$, where $\Leb(C)$ denotes the Lebesgue measure of $C$. 
    We have thus proved the desired estimate.
\end{proof}

Similar coupling techniques were already present at least in \cite[Lemma 7]{DZ95} for locally uniform permutations, and in \cite[Lemma 4.2 and Corollary 4.3]{MS11} for Mallows permutations. 
In the context of these articles, comparison with a uniform density on small boxes was possible. 
Here, to take into account the divergent behavior of our densities, we either use a global comparison with the density $\rho_{\beta,\gamma}^\nearrow$ to prove \Cref{ODG_divergence_integrable,ODG_divergence_pinched}, or make the size of the boxes depend on the number of sampled points to prove \Cref{ODG_divergence_along}.

\section{Study of reference permutons}\label{section_ref}

\subsection{Preliminaries}

The proof of \Cref{ODG_reference} hinges on the estimation of binomial variables. 
We thus state a concentration inequality usually referred to as Bernstein's inequality.
If $\mathcal{S}_n$ denotes a binomial variable of parameter $(n,p)$, then:

\begin{lemma}\label{Bernstein}
For all $t>0$:
    \begin{equation*}
    \prob{ \abs{\mathcal{S}_n - np} \ge t }
    \le 2\exp\left(-\frac{-t^2/2}{np(1-p)+t/3}\right) .
    \end{equation*}
\end{lemma}

See \cite[Equation (8)]{B62} for an easy-to-find reference and discussion on improvements, or \cite{B27} for the original one.
\bigskip

Now let us recall some asymptotics related to the sequence $(u_k)_{k\in\bN^*}$ introduced in \Cref{section_theoremes}. 
A short proof is included for completeness.
\begin{lemma}\label{equivalent_ref}
For any $\beta>1$ and $\gamma\in\bR$ we have
\begin{equation*}
    \sum_{k\ge n}k^{-\beta}\log(k+1)^\gamma
    \underset{n\to\infty}{\sim} \frac{n^{1-\beta}}{\beta-1}\log(n)^\gamma .
\end{equation*}
Moreover for any $\beta'<1$:
\begin{equation*}
    \sum_{k=1}^n k^{-\beta'}\log(k+1)^\gamma \underset{n\to\infty}{\sim} \frac{n^{1-\beta'}}{1-\beta'}\log(n)^\gamma .
\end{equation*}
\end{lemma}

\begin{proof}
First use the integral comparison:
\begin{equation*}
    \sum_{k\ge n}k^{-\beta}\log(k+1)^\gamma 
    \underset{n\to\infty}{\sim} \int_n^{+\infty}x^{-\beta}\log(x+1)^\gamma dx
\end{equation*}
and then an elementary integration by parts
\begin{align*}
    \int_n^{+\infty}x^{-\beta}\log(x+1)^\gamma dx
    &= \frac{n^{1-\beta}}{\beta-1}\log(n+1)^\gamma
    - \int_n^{+\infty}\frac{x^{1-\beta}}{1-\beta}\frac{\gamma}{x+1}\log(x+1)^{\gamma-1} dx
    \\&= \frac{n^{1-\beta}}{\beta-1}\log(n+1)^\gamma
    + \frac{\gamma}{\beta-1}\int_n^{+\infty}\frac{x^{1-\beta}}{x+1}\log(x+1)^{\gamma-1} dx.
\end{align*}
However, the following holds:
\begin{equation*}
    \int_n^{+\infty}\frac{x^{1-\beta}}{x+1}\log(x+1)^{\gamma-1} dx
    = o\left( \int_n^{+\infty}x^{-\beta}\log(x+1)^{\gamma} dx \right)
    \quad\text{as }n\to\infty.
\end{equation*}
This concludes the proof of the first assertion. 
The second one is analogous.
\end{proof}

\subsection{Proof of Proposition \ref{ODG_reference}}

In this section we fix $\beta>1$ and $\gamma\ge 0$ and prove \Cref{ODG_reference}. 
Consider $N\in\bN^*$ and write
\begin{equation*}
L_N := \LIS{\sample{N}{\mu_{\beta,\gamma}^\nearrow}}.
\end{equation*}
Let $X_1,\dots,X_N$ be i.i.d.~random variables distributed under $\mu_{\beta,\gamma}^\nearrow$.
For each $k\in\bN^*$, define
\begin{equation*}
    \mathcal{X}_{N,k} := \lbrace X_1,\dots,X_N \rbrace \cap C_k
\end{equation*}
and let $N_k$ be the cardinal of $\mathcal{X}_{N,k}$, \textit{i.e.}~the number of points appearing in box $C_{k}$. 
Each $N_k$ is a binomial variable with parameter $(N,u_{k})$, and almost surely
\begin{equation*}
    \sum_{k\ge 1}N_k = N.
\end{equation*}
Conditionally on $N_k$, the set $\mathcal{X}_{N,k}$ consists of $N_k$ uniformly random points in $C_{k}$.
Moreover, almost surely:
\begin{equation*}
    \LIS{X_1,\dots,X_N} = \sum_{k\ge 1}\LIS{\mathcal{X}_{N,k}}
\end{equation*}
thanks to the boxes being placed in an up-right fashion. 
Hence by taking expectation in the previous line, one obtains
\begin{equation*}
    \expec{ L_N } = \sum_{k\ge 1} \expec{ \ell_{N_k} }
\end{equation*}
with the notation of (\ref{def_ell_N}). 
For some integer $k_N$ to be determined, we will use the following bounds:
\begin{equation}\label{encadrement_LIS_ref}
    \sum_{k= 1}^{k_N} \expec{ \ell_{N_k} }
    \le 
    \expec{ L_N }
    \le
    \sum_{k= 1}^{k_N} \expec{ \ell_{N_k} } + N\sum_{k>k_N}u_k
\end{equation}
where the right hand side was obtained by simply bounding each $\ell_{N_k}$ for $k>k_N$ with $N_k$.
Using \Cref{Ulam}, fix an integer $n_0$ such that
\begin{equation}\label{encadrement_Ulam}
    \text{for all } n\ge n_0 , \quad
    \sqrt{n} \le \ell_n \le 3\sqrt{n}.
\end{equation}
The number $k_N$ appearing in \Cref{encadrement_LIS_ref} must be chosen big enough for the bounds to be tight, but also small enough for (\ref{encadrement_Ulam}) to be used. 
By applying Bernstein's inequality (\Cref{Bernstein}) here with an appropriate choice of parameter, we obtain for any $N,k\in\bN^*$:
\begin{multline}\label{Bernstein_Nk_ref}
    \prob{ \abs{N_k - Nu_{k}} \ge \log(N)^2\sqrt{Nu_{k}} }
    \le 2\exp\left( -\psi_{N,k} \right)
    \;\text{ where } \:
    \psi_{N,k} := \frac{\log(N)^4Nu_{k}/2}{Nu_{k}(1-u_{k})+\log(N)^2\sqrt{Nu_{k}}/3}.
\end{multline}
We will investigate the term $\psi_{N,k}$ later on, and obtain the simple upper bound (\ref{maj_psi}) for adequate values of $k$.
To apply (\ref{encadrement_Ulam}) and (\ref{Bernstein_Nk_ref}) we are looking, for each positive integer $N$, for $k_N$ satisfying
\begin{equation}\label{condition_kN}
    \text{for any } k\in\llbracket 1,k_N\rrbracket ,\quad
    Nu_{k} - \log(N)^2\sqrt{Nu_{k}} \ge n_0.
\end{equation}

\begin{lemma}\label{lem_kN}
Condition (\ref{condition_kN}) holds true for some $k_N = \underset{N\to+\infty}{\Theta}\left( \log(N)^{-4/\beta}N^{1/\beta} \right)$.
\end{lemma}

From now on, we choose $k_N$ as in \Cref{lem_kN}
(note that $k_N$ may be zero for small values of $N$).
The proof of this lemma is postponed to the end of this section.
Now, let us study the probability error term in (\ref{Bernstein_Nk_ref}). 
For any positive integer $k$ lower than or equal to $k_N$, one of the following holds:
\begin{itemize}
    \item[$\bullet$] If $Nu_{k}(1-u_{k})\ge \log(N)^2\sqrt{Nu_{k}}/3$ then
    \begin{equation*}
        \psi_{N,k} \ge \frac{\log(N)^4 Nu_{k}/2}{2Nu_{k}(1-u_{k})} \ge \frac{\log(N)^4}{4}.
    \end{equation*}
    \item[$\bullet$] Otherwise
    \begin{equation*}
        \psi_{N,k} \ge \frac{\log(N)^4 Nu_{k}/2}{2\log(N)^2\sqrt{Nu_{k}}/3}
        = \frac{3}{4}\log(N)^2\sqrt{Nu_{k}} \ge \frac{3}{4}\log(N)^2\sqrt{n_0}
    \end{equation*}
    where we used \Cref{condition_kN} in the last inequality.
\end{itemize}
Hence there exists a constant $\delta>0$ such that, for all $N\in\bN^*$:
\begin{equation}\label{maj_psi}
    \sup_{1\le k\le k_N}\exp\left( -\psi_{N,k} \right) \le \exp\left( -\delta\log(N)^2 \right).
\end{equation}
To study the bounds of (\ref{encadrement_LIS_ref}), we distinguish between the cases $\beta \in ]1,2[$ and $\beta\ge 2$.
\bigskip

\underline{First suppose $\beta \in ]1,2[$}.
Let us begin with the upper bound of (\ref{encadrement_LIS_ref}). 
On the one hand:
\begin{equation}\label{maj_reste_ref}
    N\sum_{k>k_N}u_k = \Theta\left( N k_N^{1-\beta} \log(k_N)^\gamma\right) = \Theta\left( \log(N)^{4-4/\beta+\gamma}N^{1/\beta} \right)
\end{equation}
as $N\to +\infty$, using \Cref{equivalent_ref,lem_kN}. 
On the other hand for each $k\in \llbracket 1,k_N\rrbracket$:
\begin{align*}
    \expec{\ell_{N_k}}
    &= \expec{ \ell_{N_k}\mathbf{1}_{\abs{N_k - Nu_k} < \log(N)^2\sqrt{Nu_k}} }
    + \expec{ \ell_{N_k}\mathbf{1}_{\abs{N_k - Nu_k} \ge \log(N)^2\sqrt{Nu_k}} }\\
    &\le 3\sqrt{Nu_k+\log(N)^2\sqrt{Nu_k}}
    \;+\; N \prob{ \abs{N_k - Nu_k} \ge \log(N)^2\sqrt{Nu_k} }\\
    &\le 3\sqrt{Nu_k}+3\log(N)(Nu_k)^{1/4}
    \;+\; 2N \exp\left( -\delta\log(N)^2 \right)
\end{align*}
where we used \Cref{condition_kN,encadrement_Ulam,Bernstein_Nk_ref,maj_psi}, and the inequality $\sqrt{a+b}\le \sqrt{a}+\sqrt{b}$ for any $a,b\ge 0$.
Summing and using \Cref{equivalent_ref,lem_kN}, we get
\begin{align}\label{maj_somme_ref}
    \sum_{k= 1}^{k_N} \expec{ \ell_{N_k} }
    &\le 3\sqrt{N}\sum_{k=1}^{k_N}\sqrt{u_k}
    + 3\log(N)N^{1/4}\sum_{k=1}^{k_N}u_k^{1/4}
    + 2N k_N \exp\left( -\delta\log(N)^2 \right)
    \\&\nonumber = 3\sqrt{N}
    \Theta\left( k_N^{1-\beta/2}\log(k_N)^{\gamma/2} \right)
    + 3\log(N)N^{1/4}
    \Theta\left( k_N^{1-\beta/4}\log(k_N)^{\gamma/4} \right)
    + o(1)
    \\&\nonumber = \Theta\left( 
    \log(N)^{2-4/\beta+\gamma/2}
    N^{1/\beta} 
    \right).
\end{align}
as $N\to +\infty$. This last upper bound along with (\ref{maj_reste_ref}) yields, in (\ref{encadrement_LIS_ref}):
\begin{equation}\label{maj_LN_ref}
    \expec{L_N} = 
    \cO\left( \log(N)^{4-4/\beta+\gamma}N^{1/\beta}\right) 
    \quad\text{as }N\to\infty.
\end{equation}\medskip

Now let us turn to the lower bound of (\ref{encadrement_LIS_ref}), for which the calculations are simpler.
For any $k\in\llbracket1,k_N\rrbracket$:
\begin{align*}
    \expec{\ell_{N_k}}
    \ge \expec{ \ell_{N_k}\mathbf{1}_{\abs{N_k - Nu_k} < \log(N)^2\sqrt{Nu_k}} }
    \ge \sqrt{n_0}
    \left(1-2\exp(-\psi_{N,k})\right)
\end{align*}
using \Cref{condition_kN,encadrement_Ulam,Bernstein_Nk_ref}.
Then by summing and using \Cref{lem_kN} and Eq.~(\ref{maj_psi}):
\begin{align*}
    \expec{L_N}
    \ge k_N\sqrt{n_0}\left(1-2\exp\left( -\delta\log(N)^2 \right)\right)
    = \Omega\left( 
    \log(N)^{-4/\beta}
    N^{1/\beta} \right)
\end{align*}
as $N\to+\infty$. 
This lower bound, along with (\ref{maj_LN_ref}), concludes the proof of \Cref{ODG_reference} when $\beta\in]1,2[$.
\bigskip

\underline{Now suppose $\beta \ge 2$.}
The upper bound is very similar to the case $\beta\in]1,2[$, but with the appropriate asymptotics. 
Namely for any $\beta'\ge 1$ and $\gamma'\ge 0$ one has:
\begin{equation}\label{equivalent_bertrand_integrable}
    \sum_{k=1}^n k^{-\beta'}\log(k+1)^{\gamma'} =\cO\left( \log(n)^{1+\gamma'} \right) 
    \quad\text{as }n\to\infty.
\end{equation}
On the one hand (\ref{maj_reste_ref}) still holds and we can thus write
\begin{equation}\label{maj_reste_ref_grand_expo}
    N\sum_{k>k_N}u_k = \OT\left( \sqrt{N} \right)\quad\text{as }N\to\infty.
\end{equation}
On the other hand the first line of (\ref{maj_somme_ref}) is still valid and we obtain, as $N\to\infty$:
\begin{align*}
    \sum_{k= 1}^{k_N} \expec{ \ell_{N_k} }
    &\le 3\sqrt{N}\sum_{k=1}^{k_N}\sqrt{u_k}
    + 3\log(N)N^{1/4}\sum_{k=1}^{k_N}u_k^{1/4}
    + 2N k_N \exp\left( -\delta\log(N)^2 \right)
    \\&\le \sqrt{N}
    \cO\left(\log(N)^{1+\gamma/2}\right) + o(1)
    \\&\quad + \log(N)N^{1/4}
    \max\left( \cO\left(\log(N)^{1+\gamma/4}\right) , \Theta\left( k_N^{1-\beta/4}\log(k_N)^{\gamma/4} \right) \right)
    \\&= \cO\left( \log(N)^{1+\gamma/2}\sqrt{N} \right)
    \\&\quad + \max\left(
    \cO\left( \log(N)^{2+\gamma/4}N^{1/4} \right) , \Theta\left( \log(N)^{2-4/\beta+\gamma/4}N^{1/\beta} \right)
    \right)
    \\&= \OT\left(\sqrt{N}\right)
\end{align*}
using \Cref{lem_kN} and Eq.~(\ref{equivalent_bertrand_integrable}) and distinguishing between the cases $\beta\ge4$ and $\beta<4$ on the second line. 
Along with (\ref{maj_reste_ref_grand_expo}), this yields the desired upper bound when injected in (\ref{encadrement_LIS_ref}).
\medskip

The lower bound, on the contrary, requires no calculation. 
Indeed, bound below $\rho_{\beta,\gamma}^\nearrow$ by $u_1f$ where $f$ denotes the uniform density on the square $C_1$. 
Since permutations sampled from the density $f$ are uniform, we deduce from \Cref{Ulam} and \Cref{lem_maj_density} that
\begin{equation*}
    \expec{ \LIS{\sample{N}{\mu_{\beta,\gamma}^\nearrow}} }
    =\underset{N\to\infty}{\Omega}\left(\ell_{\lfloor u_1N \rfloor}\right)
    =\underset{N\to\infty}{\Omega}\left(\sqrt{N}\right).
\end{equation*}\medskip

All that is left for the proof of \Cref{ODG_reference} to be complete is the previously announced lemma about $k_N$.

\begin{proof}[Proof of \Cref{lem_kN}]
Let $N\in\bN^*$. 
For each integer $k$:
\begin{align*}
    Nu_k - \log(N)^2\sqrt{Nu_k} \ge n_0
    &\;\Leftrightarrow\;
    \left\{
    \begin{array}{ll}
         Nu_k\ge n_0 ;\\
         Nu_k-n_0 \ge \log(N)^2\sqrt{Nu_k} ;
    \end{array}
    \right.\\
    &\;\Leftrightarrow\;
    \left\{
    \begin{array}{ll}
         Nu_k\ge n_0 ;\\
         N^2 u_k^2 + n_0^2 - 2Nu_k n_0 \ge \log(N)^4Nu_k ;
    \end{array}
    \right.\\
    &\;\Leftrightarrow\;
    \left\{
    \begin{array}{ll}
         u_k\ge n_0/N ;\\
         N^2 u_k^2 - N\left( 2n_0 + \log(N)^4 \right)u_k + n_0^2 \ge 0.
    \end{array}
    \right.
\end{align*}
This last polynomial in the variable $u_k$ has discriminant $\Delta_N = N^2\left(\log(N)^8+4n_0\log(N)^4\right) \ge 0$. Let $x_{N}$ be its greatest root. Then
\begin{equation*}
    x_{N} = \frac{N\left( 2n_0+\log(N)^4\right) +\sqrt{\Delta_N}}{2N^2} \underset{N\to\infty}{\sim} \frac{\log(N)^4}{N}.
\end{equation*}
A sufficient condition for $Nu_k - \log(N)^2\sqrt{Nu_k} \ge n_0$ to hold is $u_k\ge \max\left( n_0/N , x_N \right)$.
However
\begin{equation*}
    u_k = \frac{1}{Z_{\beta,\gamma}}k^{-\beta}\log(k+1)^\gamma \ge c_0 k^{-\beta}
\end{equation*}
for $c_0 := \log(2)^\gamma /Z_{\beta,\gamma}$, so a sufficient condition is
\begin{equation*}
    k \le c_0^{1/\beta}
    \max\left( n_0/N , x_N \right)^{-1/\beta}.
\end{equation*}
Hence the announced estimate for $k_N$.
\end{proof}

\section{Study of densities diverging at a single point}\label{section_last_proofs}

\subsection{Lower bound of Theorem \ref{ODG_divergence_pinched}}\label{section_preuve_lower}

This bound is quite direct thanks to \Cref{lem_maj_density} and the previous study of $\mu_\beta^\nearrow$. 
We will use the notation of \Cref{section_theoremes} with the same $\beta$ as in \Cref{ODG_divergence_pinched} and $\gamma=0$. 
Studying $\rho$ on the boxes $(C_n)_{n\in\bN^*}$ will be enough to obtain the desired lower bound.
Fix $\varepsilon > 0$ and some rank $m_0\in\bN^*$ such that
\begin{equation*}
    \text{for all } n\ge m_0 \text{ and } (x,y)\in C_n ,\quad
    \rho(x,y)\ge \varepsilon
    d^{\frac{\beta}{1-\beta}}
    \exp\left(
    -c |x-y| d^{\frac{\beta}{1-\beta}}
    \right).
\end{equation*}
Recall the notation $d:=\dist((x,y),(1,1))$ for $(x,y)\in\carre$
and write $R_{n} := \sum_{k>n}u_{k}$ for all $n\in\bN$.
Note that, for $(x,y)\in C_n$:
\begin{equation*}
    2R_n \le d \le 2R_{n-1} 
    \quad\text{where}\quad
    R_{n-1} , R_n = \underset{n\to\infty}{\Theta}\left( n^{1-\beta} \right)
\end{equation*}
by \Cref{equivalent_ref}, and
\begin{equation*}
    |x-y| \le u_n = \underset{n\to\infty}{\Theta}\left( n^{-\beta} \right).
\end{equation*}
As a consequence, for potentially different values of $\varepsilon>0$ and $m_0\in\bN^*$ we get
\begin{equation*}
    \text{for all } n\ge m_0 \text{ and } (x,y)\in C_n ,\quad
    \rho(x,y)
    \ge \varepsilon \rho_\beta^\nearrow(x,y).
\end{equation*}
Write $g$ for the probability density on $\cup_{n\ge m_0}C_n$ proportional to $\rho_\beta^\nearrow$. 
Then by \Cref{lem_maj_density}:
\begin{equation*}
    \expec{ \LIS{\sample{N}{\mu_\rho}} } = \underset{N\to\infty}{\Omega}\left( \expec{ \LIS{\sample{\lfloor\varepsilon N\rfloor}{\mu_g}} } \right).
\end{equation*}
Moreover we can obtain
\begin{equation*}
    \expec{ \LIS{\sample{N}{\mu_g}} } 
    = \underset{N\to\infty}{\ThetaT}\left( N^{1/\beta} \right)
\end{equation*}
with the same proof as for the reference permuton $\mu_\beta^\nearrow$ (start every index at $m_0$ instead of $1$). 
Finally:
\begin{equation*}
    \expec{ \LIS{\sample{N}{\mu_\rho}} } = \underset{N\to\infty}{\OmegaT}\left( N^{1/\beta}\right).
\end{equation*}

\subsection{Upper bound of Theorem \ref{ODG_divergence_pinched}}

This bound is more subtle than the previous one. 
Indeed, long increasing subsequences could appear outside of the boxes used in \Cref{section_preuve_lower}.
Our solution comes in two steps: 
first consider slightly bigger boxes, and then add an overlapping second sequence of boxes to make sure a whole neighborhood of the diagonal is covered.

We will mainly use the notation of \Cref{section_theoremes} with the number $\beta$ considered in \Cref{ODG_divergence_pinched} and any negative number
$\gamma < 1-\beta$. 
In addition to the boxes $(C_n)_{n\in\bN^*}$, define for all $n\in\bN^*$:
\begin{equation*}
    D_{n+1} := \left[ S_n - \frac{u_{n+1}}{2} , S_n + \frac{u_{n+1}}{2} \right]^2
    \quad\text{and}\quad
    E_n := [S_{n-1},1]^2 \setminus \left(
    [S_n,1]^2 \cup C_n \cup D_{n+1} 
    \right),
\end{equation*}
and their unions:
\begin{equation*}
    E := \bigcup_{n\ge 1}E_n
    \quad , \quad
    C := \bigcup_{n\ge 1}C_n
    \quad , \quad
    D := \bigcup_{n\ge 1}D_{n+1} .
\end{equation*}

\begin{figure}
\begin{center}
\includegraphics[scale=0.4]{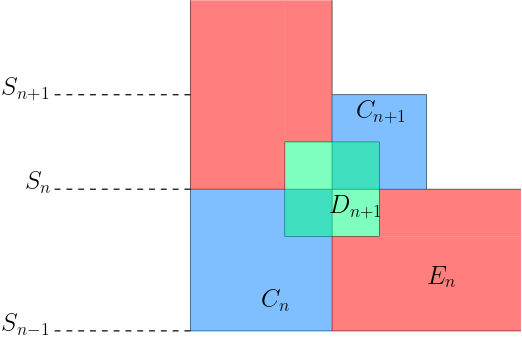}
\caption{The areas used to study density $\rho$ in the upper bound of \Cref{ODG_divergence_pinched}.}\label{fig_decoupage}
\end{center}
\end{figure}

See \Cref{fig_decoupage} for a visual representation. 
Notice how these three areas cover the whole unit square.
Let us check that $\rho$ is small outside the diagonal neighborhood $C\cup D$:
\begin{equation*}
    \text{for all } n\in\bN^* \text{ and } (x,y) \in E_n ,\quad
    |x-y| \ge u_{n+1}/2
    \:\text{ and }\:
    R_n \le \dist\big( (x,y),(1,1) \big) \le 2R_{n-1}
\end{equation*} 
where $R_{n} := \sum_{k>n}u_{k}$.
Then using \Cref{equivalent_ref}, we get as $n\to\infty$ uniformly in $(x,y)\in E_n$:
\begin{equation*}
    |x-y| = \Omega\left( n^{-\beta}\log(n)^\gamma \right)
    \quad\text{and}\quad
    \dist\big( (x,y),(1,1) \big) = \Theta\left( n^{1-\beta}\log(n)^\gamma \right).
\end{equation*}
Our hypothesis on $\rho$ now rewrites
\begin{align*}
    \rho (x,y)
    &= \cO\left(
    \left(
    n^{1-\beta} \log(n)^{\gamma}
    \right)^{\frac{\beta}{1-\beta}}
    \exp\left(
    - \Omega\left(
    n^{-\beta}\log(n)^{\gamma}
    \left(
    n^{1-\beta} \log(n)^{\gamma}
    \right)^{\frac{\beta}{1-\beta}}
    \right)
    \right)
    \right)
    \\&= \cO\left(
    n^\beta \log(n)^{\frac{\gamma\beta}{1-\beta}}
    \exp\left(
    - \Omega\left(
    \log(n)^{\frac{\gamma}{1-\beta}}
    \right)
    \right)
    \right)
    \\&= o(1)
\end{align*}
since $\frac{\gamma}{1-\beta} > 1$ by choice of $\gamma$. 
In particular $\rho$ is bounded on $E$, and it remains to study it on areas $C$ and $D$. 
Using \Cref{equivalent_ref} and bounding the exponential term by $1$, we get as $n\to\infty$ uniformly in $(x,y)\in C_n\cup D_{n+1}$:
\begin{align*}
    \rho(x,y) = \cO\left(
    n^\beta \log(n)^{\frac{\gamma\beta}{1-\beta}}
    \right).
\end{align*}
Consequently
\begin{equation*}
    \int_{C_n}\rho(x,y)dxdy = \cO\left(
    n^{-2\beta}\log(n)^{2\gamma} n^\beta \log(n)^{\frac{\gamma\beta}{1-\beta}}
    \right) = \cO\left(
    n^{-\beta} \log(n)^{\gamma\left( 2-\frac{\beta}{\beta-1} \right)}
    \right)
    \quad\text{as }n\to\infty
\end{equation*}
and likewise
\begin{equation*}
    \int_{D_{n+1}}\rho(x,y)dxdy = \cO\left(
    n^{-\beta} \log(n)^{\gamma\left( 2-\frac{\beta}{\beta-1} \right)}
    \right)
    \quad\text{as }n\to\infty.
\end{equation*}
Define $\gamma' := \gamma\left( 2-\frac{\beta}{\beta-1} \right) >0$. 
The previous calculations show that we can find a bound
\begin{equation*}
    \rho \le M(f+g+h)
\end{equation*}
for some $M>0$, $f$ the uniform density on $\carre$, $g$ a probability density on $C$ attributing uniform mass proportional to $n^{-\beta}\log(n+1)^{\gamma'}$ to each $C_n$, and $h$ a probability density on $D$ giving uniform mass proportional to $n^{-\beta}\log(n+1)^{\gamma'}$ to each $D_{n+1}$. 
Thus by \Cref{lem_maj_sum_density} it suffices to bound the quantities
\begin{equation*}
    \expec{\LIS{\sample{N}{\mu_f}}}
    ,\quad
    \expec{\LIS{\sample{N}{\mu_g}}}
    ,\quad
    \expec{\LIS{\sample{N}{\mu_h}}} .
\end{equation*}
The first term is nothing but the uniform case, so it behaves as $\Theta\left(\sqrt{N}\right)$.
Let us turn to the second term. 
Since the sampled permutations of our reference permutons only depend on the masses attributed to each box and not the sizes of these boxes, sampled permutations from $\mu_g$ have the same law as sampled permutations from $\mu_{\beta,\gamma'}^\nearrow$ (see \Cref{section_discussions}; 
$\mu_{\beta,\gamma'}^\nearrow$ is the permuton associated to the pre-permuton $\mu_g$). 
Hence by \Cref{ODG_reference}, this term behaves as $\ThetaT\left(N^{1/\beta}\right)$. 
The third term is handled in the same way. 
Finally:
\begin{equation*}
    \expec{\LIS{\sample{N}{\mu_\rho}}} = \underset{N\to\infty}{\OT}\left(
    N^{1/\beta}
    \right).
\end{equation*}

\subsection{Proof of Theorem \ref{ODG_divergence_integrable}}

This section is devoted to the proof of \Cref{ODG_divergence_integrable}. 
We thus consider $\alpha>-2$ and suppose $\rho$ is as in the theorem. 
Since we want an upper bound on the longest increasing subsequences, we need to find appropriate areas to bound $\rho$ on. 
To this aim, define $\beta>1$ by
\begin{equation}\label{def_beta}
    -\alpha = 2 - \frac{1}{\beta - 1}
    \quad\textit{ i.e.}\quad
    \alpha(1-\beta) + 1 - 2\beta = -2.
\end{equation}
We use the notation of \Cref{section_theoremes} for this value of $\beta$ and $\gamma=0$. 
We shall bound $\rho$ on the boxes $(C_n)_{n\in\bN^*}$ as well as on the adjacent rectangles:
\begin{equation*}
    \text{for all } n\in\bN^*,\quad
    D_n^{(1)} := [S_{n-1},S_n]\times [S_n,1]
    \quad\text{and}\quad
    D_n^{(2)} := [S_n,1]\times [S_{n-1},S_n] .
\end{equation*}
The sequences $(C_n)_{n\in\bN^*},(D_n^{(1)})_{n\in\bN^*},(D_n^{(2)})_{n\in\bN^*}$ form a partition of the unit square. 
As in the upper bound of \Cref{ODG_divergence_pinched}, we need to compute the masses attributed by $\rho$ to each of these boxes. 
For this notice that 
\begin{equation*}
    \text{for all } n\in\bN^* \text{ and } (x,y)\in C_n\cup D_n^{(1)}\cup D_n^{(2)},\quad
    R_n \le \dist((x,y),(1,1)) \le 2R_{n-1}
\end{equation*}
where $R_{n} := \sum_{k>n}u_{k}$.
Hence
\begin{equation*}
    \int_{C_n}\rho(x,y)dxdy = \underset{}{\Theta}\left( n^{\alpha(1-\beta)}n^{-2\beta} \right)
    = \underset{}{\Theta}\left( n^{-3} \right)
    \quad\text{as }n\to\infty
\end{equation*}
and
\begin{equation*}\label{mass_on_side_rectangles}
    \text{for } i=1,2 ,\quad 
    \int_{D_n^{(i)}}\rho(x,y)dxdy = \underset{}{\Theta}\left( n^{\alpha(1-\beta)}n^{-\beta}n^{1-\beta} \right)
    = \underset{}{\Theta}\left( n^{-2} \right)
    \quad\text{as }n\to\infty
\end{equation*}
by (\ref{def_beta}). Using \Cref{lem_maj_sum_density}, it suffices to bound the quantities
\begin{equation*}
    \expec{\LIS{\sample{N}{\mu_f}}}
    ,\quad
    \expec{\LIS{\sample{N}{\mu_g}}}
    ,\quad
    \expec{\LIS{\sample{N}{\mu_h}}}
\end{equation*}
where $f$ is the probability density on $C$ attributing uniform mass proportional to $n^{-3}$ to each $C_n$ and $g$ (resp.~$h$) is the probability density on $D^{(1)}$ (resp.~$D^{(2)}$) attributing uniform mass proportional to $n^{-2}$ to each $D_n^{(1)}$ (resp.~$D_n^{(2)}$).
Considering the reference permuton $\mu_3^\nearrow$ of parameter $(3,0)$, \Cref{ODG_reference} tells us 
\begin{equation*}
    \expec{\LIS{\sample{N}{\mu_3^{\nearrow}}}} = \underset{N\to\infty}{\ThetaT}\left( \sqrt{N} \right).
\end{equation*}
Since $\mu_3^\nearrow$ attributes the same masses to the boxes of its support as density $f$ attributes to its own, sampled permutations from both pre-permutons have same law (the same remark as in the upper bound of \Cref{ODG_divergence_pinched} holds; 
$\mu_3^\nearrow$ is the permuton associated to the pre-permuton $\mu_f$). 
Consequently:
\begin{equation*}
    \expec{\LIS{\sample{N}{\mu_f^{\nearrow}}}} = \underset{N\to\infty}{\ThetaT}\left( \sqrt{N} \right).
\end{equation*}
The case of density $g$ is similar but with a slight alteration.
Indeed, considering the reference permuton $\mu_2^\nearrow$ of parameter $(2,0)$, \Cref{ODG_reference} tells us 
\begin{equation}\label{odg_ref_beta2}
    \expec{\LIS{\sample{N}{\mu_2^{\nearrow}}}} = \underset{N\to\infty}{\ThetaT}\left( \sqrt{N} \right).
\end{equation}
Note that $\mu_2^\nearrow$ attributes the same masses to the boxes of its support as density $g$ attributes to its own. 
A key difference here is that the rectangle boxes $D_n^{(1)}$ are not placed in an up-right manner inside the unit square, so sampled permutations from permuton $\mu_2^{\nearrow}$ and density $g$ do not have the same law. 
To work around this problem, we can use an appropriate coupling. 
Take random i.i.d.~points $X_1,\dots,X_N$ distributed under density $g$. 
Consider, for each $n\in\bN^*$, the affine transformation $a_n$ mapping $D_n^{(1)}$ to $C_n$ and assemble them into a function $a$ from $\cup_{n\ge1}D_n^{(1)}$ to $\cup_{n\ge1}C_n$.
Then the image points $a(X_1),\dots,a(X_N)$ are i.i.d.~under the measure $\mu_2^\nearrow$. 
Moreover, each increasing subset of $\lbrace X_1,\dots,X_N\rbrace$ is mapped to an increasing subset of $\lbrace a(X_1),\dots,a(X_N)\rbrace$. 
This coupling argument shows that $\LIS{\sample{N}{\mu_g}}$ is stochastically dominated by $\LIS{\sample{N}{\mu_2^\nearrow}}$, and (\ref{odg_ref_beta2}) then implies:
\begin{equation*}
    \expec{\LIS{\sample{N}{\mu_g}}} = \underset{N\to\infty}{\OT}\left( \sqrt{N} \right).
\end{equation*}
Density $h$ is handled in the same way. 
Hence:
\begin{equation*}
    \expec{\LIS{\sample{N}{\mu_\rho}}} = \underset{N\to\infty}{\OT}\left( \sqrt{N} \right).
\end{equation*}
To conclude the proof of \Cref{ODG_divergence_integrable}, the lower bound is obtained as a direct consequence of \Cref{lem_maj_density} using the uniform case (see the proof of \Cref{coro_cont}).

\section{Study of densities diverging along the diagonal}\label{section_new_proofs}

\subsection{Lower bound of Theorem \ref{ODG_divergence_along}}

From now on we consider a density $\rho$ satisfying the hypothesis of \Cref{ODG_divergence_along} for some exponent $\alpha\in]-1,0[$. 
As explained in \Cref{section_methodes}, the idea is to slice the unit square into small boxes and investigate the number of sampled points appearing in appropriate increasing sequences of boxes.
Let $N\in\bN^*$ and take random i.i.d.~points $X_1,\dots,X_N$ distributed under the density $\rho$. Set
\begin{equation*}
    b_N := \left\lfloor N^{1/(\alpha+2)} \right\rfloor,
\end{equation*}
and define a family of $b_N^2$ identical boxes by
\begin{equation*}
    \text{for all }(i,j)\in\llbracket1,b_N\rrbracket^2,\quad
    C_{i,j} := \left[ \frac{(i-1)}{b_N} , \frac{i}{b_N}\right]\times\left[ \frac{(j-1)}{b_N} , \frac{j}{b_N}\right].
\end{equation*}
This covering of the unit square will be useful for the upper bound, while the lower bound aimed for in this section only requires using the increasing sequence of boxes $(C_{k,k})_{k\in\llbracket1,b_N\rrbracket}$.
More precisely, we make use of the inequality
\begin{equation}\label{un_point_chaque_boite}
    \LIS{X_1,\dots,X_N} \ge \sum_{k=1}^{b_N}\mathbf{1}_{N_k\ge 1}
\end{equation}
where each $N_k$ denotes the number of points among $X_1,\dots,X_N$ in $C_{k,k}$.
Thanks to the hypothesis made on $\rho$, we can fix $\delta,\varepsilon>0$ such that
\begin{equation*}
    \text{for all }(x,y)\in\carre
    \text{ satisfying }\vert x-y\vert <\delta,
    \quad
    \rho(x,y)\ge \varepsilon\vert x-y\vert^\alpha.
\end{equation*}
Now suppose $N$ is large enough to have $b_N^{-1}<\delta$ and compute, for any $k\in\llbracket1,b_N\rrbracket$:
\begin{multline*}
    m_k := 
    \int_{C_{k,k}}\rho(x,y) dxdy
    \ge \int_{C_{k,k}}\varepsilon\vert x-y\vert^\alpha dxdy
    = \varepsilon\int_0^{b_N^{-1}}dx\int_{-x}^{b_N^{-1}-x}dz\vert z\vert^\alpha
    \\= \varepsilon\int_0^{b_N^{-1}}dx\left(\frac{x^{\alpha+1}}{\alpha+1} + \frac{(b_N^{-1}-x)^{\alpha+1}}{\alpha+1}\right)
    = \frac{2\varepsilon}{(\alpha+1)(\alpha+2)}b_N^{-(\alpha+2)}
    = \underset{N\to\infty}{\Omega}\left(\frac{1}{N}\right).
\end{multline*}
Hence there exists $\eta>0$ such that for all $N\in\bN^*$ and $k\in\llbracket1,b_N\rrbracket$, $m_k\ge \frac{\eta}{N}$.
Since $N_k$ follows a binomial law of parameter $(N,m_k)$, we deduce:
\begin{equation*}
    \bP(N_k=0)=(1-m_k)^N \le (1-\eta/N)^N \cv{N\to\infty} e^{-\eta}<1.
\end{equation*}
Consequently there exists $p_0>0$ such that for any large enough $N\in\bN^*$ and all $k\in\llbracket1,b_N\rrbracket$, $\bP(N_k\ge 1)\ge p_0$.
Hence for large enough $N$, using \Cref{un_point_chaque_boite}:
\begin{equation*}
    \expec{\LIS{\sample{N}{\mu_\rho}}}
    \ge \expec{ \sum_{k=1}^{b_N}\mathbf{1}_{N_k\ge 1} }
    \ge \sum_{k=1}^{b_N} p_0
    = p_0 b_N
    = \underset{N\to\infty}{\Omega}\left(N^{1/(\alpha+2)}\right).
\end{equation*}

\subsection{Upper bound of Theorem \ref{ODG_divergence_along}}

We use the same notation as in the previous section, but this time we investigate the whole grid $(C_{i,j})_{i,j\in\llbracket1,b_N\rrbracket}$. 
Say a sequence of distinct boxes
$\mathcal{C}=(C_{i_1,j_1},\dots,C_{i_n,j_n})$
is increasing whenever
\begin{equation*}
    \text{for all }k\in \llbracket1,n-1\rrbracket,\quad
    i_k\le i_{k+1}
    \text{ and }
    j_k\le j_{k+1}.
\end{equation*}
When this happens, one has $n < 2b_N$. 
Indeed, when browsing the sequence, each coordinate increases at most $b_N-1$ times.
\smallskip

Write $\mathcal{X} := \lbrace X_1,\dots,X_N\rbrace$.
Then, for any box $C$, denote by $\mathcal{X}_C$ the set of points in $\mathcal{X}$ appearing in $C$ and, for any increasing sequence of boxes $\mathcal{C}$, denote by $\mathcal{X}_\mathcal{C}$ the set of points in $\mathcal{X}$ appearing in some box of $\mathcal{C}$.
We aim to make use of the inequality
\begin{equation}\label{nombre_points_chaque_boite}
    \LIS{\mathcal{X}} \le \sup_{\mathcal{C}\text{ increasing sequence of boxes}} \abs{\mathcal{X}_\mathcal{C}}
\end{equation}
since the family of boxes occupied by an increasing subset of points necessarily rearranges as an increasing sequence of boxes. 
Now, thanks to the hypothesis made on $\rho$, let $M>0$ be such that
\begin{equation*}
    \text{for all }(x,y)\in\carre,\quad
    \rho(x,y)\le M |x-y|^\alpha.
\end{equation*}
Since this latter function puts more mass on the diagonal boxes than the outside ones, we have for any $i,j\in\llbracket1,b_N\rrbracket$:
\begin{multline*}
    \int_{C_{i,j}}\rho(x,y)dxdy
    \le \int_{C_{i,j}}M |x-y|^\alpha dxdy
    \le M\int_0^{b_N^{-1}}\int_0^{b_N^{-1}} |x-y|^\alpha dxdy
    \\=\frac{2M}{(\alpha+1)(\alpha+2)}b_N^{-(\alpha+2)}
    = \underset{N\to\infty}{\cO}(1/N).
\end{multline*}
Thus there exists $M'>0$ such that, for large enough $N$, each variable $\abs{\mathcal{X}_{C_{i,j}}}$ is stochastically dominated by the law $\mathrm{Bin}(N,M'/N)$. 
Additionally \Cref{Bernstein} yields, denoting by $\mathcal{S}_N$ a random variable of law $\mathrm{Bin}(N,M'/N)$:
\begin{multline*}
    \prob{ \abs{\mathcal{S}_N - M'} \ge \log(N)^2\sqrt{M'} }
    \le 2\exp\left(-\psi_N\right)
    \\\text{where}\quad
    \psi_N = \frac{
    \log(N)^4 M' /2
    }{
    M' (1-M'/N) + \log(N)^2 \sqrt{M'}/3
    } = \underset{N\to\infty}{\Theta}(\log(N)^2).
\end{multline*}
This inequality, along with the aforementioned stochastic domination, implies that for large enough $N$:
\begin{equation*}
    \prob{ \forall (i,j) ,\; \abs{\mathcal{X}_{C_{i,j}}} \le M'+\log(N)^2\sqrt{M'} } \ge 1 - 2b_N^{2}\exp(-\psi_N) .
\end{equation*}
Hence, using the fact that an increasing sequence of boxes contains at most $2b_N$ boxes:
\begin{equation*}
    \prob{ \text{for any increasing sequence of boxes }\mathcal{C} ,\; \abs{\mathcal{X}_\mathcal{C}} \le 2b_N\left(M'+\log(N)^2\sqrt{M'}\right) } 
    \ge 1 - 2b_N^{2}\exp(-\psi_N)
\end{equation*}
and then, by \Cref{nombre_points_chaque_boite}:
\begin{equation*}
    \prob{ \LIS{\mathcal{X}} \le 2b_N\left(M'+\log(N)^2\sqrt{M'}\right) }
    \ge 1 - 2b_N^{2}\exp(-\psi_N).
\end{equation*}
To conclude the proof of \Cref{ODG_divergence_along}, it suffices to write:
\begin{equation*}
    \expec{\LIS{\sample{N}{\mu_\rho}} }
    \le 2b_N\left(M'+\log(N)^2\sqrt{M'}\right) + 2b_N^{2}N\exp(-\psi_N)
    = \OT\left( N^{1/(\alpha+2)} \right) .
\end{equation*}

\section*{Acknowledgements}

The author would like to express his sincere thanks to Valentin F{\'e}ray for his constant support and enlightening discussions, as well as all members of the IECL who contribute to a prosperous environment for research in mathematics.

\bibliographystyle{alphaurl}
\bibliography{bibli}

\end{document}